\documentclass[11pt]{amsart}
\usepackage{amsfonts}
\usepackage{amsmath}
\usepackage{amssymb}
\usepackage{graphicx}
\usepackage{graphicx}
\usepackage{amscd}
\usepackage{amsmath}
\usepackage{amssymb}
\usepackage[latin1]{inputenc}
\usepackage[colorlinks=true]{hyperref}
\usepackage{mathrsfs}

\newtheorem{theorem}{Theorem}[section]

\newtheorem{lemma}[theorem]{Lemma}

\theoremstyle{definition}

\newtheorem{example}{Example}[section]
\newtheorem{remark}{Remark}[section]
\numberwithin{equation}{section}

\begin{document}
\title{Multiple positive solutions to a fourth order boundary value problem}
\author{Alberto Cabada, Radu Precup, Lorena Saavedra, Stepan Tersian}
\maketitle

\begin{abstract}
We study the existence and multiplicity of positive solutions for a
nonlinear fourth-order two-point boundary value problem. The approach is
based on critical point theorems in conical shells, Krasnoselskii's
compression-expansion theorem, and unilateral Harnack type inequalities.
\end{abstract}

\noindent Keywords: Fourth order differential equation; boundary value
problem; positive solution; critical point; fixed point.

\noindent Mathematics Subject Classification: 34B18, 47J30

\section{Introduction}

The fourth-order boundary value problems appear in the elasticity theory
describing stationary states of the deflection of an elastic beam. In last
decade a lot of studies are devoted to the existence of positive solutions
for such problems, applying Leray-Schauder continuation method, the
topological degree theory, the fixed point theorems on cones, the critical
point theory or the lower and upper solution method (see, for example, \cite%
{CaT, CG, CFM, GKKY, ES, LZL, LZ, Ma, YCY, WIF}).

In this article, we study the existence and multiplicity of positive
solutions for nonlinear fourth-order two-point boundary value problem with
cantilever boundary conditions. Consider the fourth-order boundary value
problem
\begin{equation}
\left\{
\begin{array}{ll}
u^{(4)}(t)-\,f(t,u(t))=0, & 0<t<1, \\
u(0)=u^{\prime }(0)=u^{\prime \prime }(1)=u^{\prime \prime \prime }(1)=0, &
\end{array}%
\right.  \label{p}
\end{equation}%
where the function $f\colon \lbrack 0,1]\times \mathbb{R}\rightarrow \mathbb{%
R}$ is continuous, and $f(t,\mathbb{R}_{+})\subset \mathbb{R}_{+}$ for all $%
t\in \lbrack 0,1].$

Our approach is based on critical point theorems for functionals in conical
shells (see \cite{Pre, Pre2}) and Krasnoselskii's compression-expansion
theorem. As one can see along the paper, the arguments developed here can be
applied to other boundary value problems associated to fourth and sixth
order differential equations. Because the estimates are connected with the
specific boundary conditions, we concentrate on the model problem (1.1).

The paper is organized as follows. In Section 2 we give formulation of
critical point theorems in conical shells and Krasnoselskii's
Compression-Expansion Theorem. We present also the variational formulation
of the problem. In Section 3, the main existence and multiplicity results
Theorems 3.2, 3.3, 3.5 and 3.6 are formulated and proved. Their proofs are
based on the mentioned above theorems and inequalities proved in Lemma 3.1
and Lemma 3.4. In order to illustrate the obtained results, two examples are
given.

\section{Preliminaries}

\subsection{Critical point theorems in conical shells}

In this subsection we introduce the results given in \cite{Pre} which we are
going to apply to the fourth order problem (\ref{p}).

We consider two real Hilbert spaces, $X$ with inner product and norm $(\cdot
,\cdot ),$ $|\cdot |,$ $H$ with inner product and norm $\left\langle \cdot
,\cdot \right\rangle ,$ $\left\Vert \cdot \right\Vert ,$ and we assume that $%
X\subset H$ with continuous injection. We identify $H$ to its dual $%
H^{\prime }$ and we obtain $X\subset H\equiv H^{\prime }\subset X^{\prime }.$
By $\left\langle \cdot ,\cdot \right\rangle $ we also denote the duality
between $X$ and $X^{\prime },$ i.e. $\left\langle x^{\ast },x\right\rangle
=x^{\ast }(x)$ for $x^{\ast }\in X^{\prime }$ and $x\in X.$ If $x^{\ast }\in
H,$ then $\left\langle x^{\ast },x\right\rangle $ is exactly the scalar
product in $H$ and $\left\langle x^{\ast },x\right\rangle =(x^{\ast }, x).$

We also consider a cone in $X,$ i.e. a convex closed nonempty set $K,$ $%
K\neq \left\{ 0\right\} ,\ $with $\lambda \,u\in K$ for every $u\in K$ and $%
\lambda \geq 0,$ and $K\cap \left( -K\right) =\left\{ 0\right\} .$ Let $\phi
\in K\setminus \left\{ 0\right\} $ be a fixed element with $\left\vert \phi
\right\vert =1.$ Then, for all numbers $R_{0},R_{1}$ with $%
0<R_{0}<\left\Vert \phi \right\Vert R_{1},$ there is $\mu >0$ such that $%
\left\Vert \mu \phi \right\Vert >R_{0}$ and $\left\vert \mu \phi \right\vert
<R_{1}.$ Denote by $K_{R_{0}R_{1}}$ the conical shell
\begin{equation*}
K_{R_{0}R_{1}}=\left\{ u\in K\ :\ \left\Vert u\right\Vert \geq R_{0},\
|u|\leq R_{1}\right\} .
\end{equation*}%
Clearly $\mu \phi $ is an interior point of $K_{R_{0}R_{1}},$ in the sense
that $\left\Vert \mu \phi \right\Vert >R_{0}$ and $\left\vert \mu \phi
\right\vert <R_{1}.$

Let $L$ be the continuous linear operator from $X$ to $X^{\prime },$ given
by
\begin{equation*}
(u,v)=\left\langle L\,u,v\right\rangle \quad \text{for all \ }u,\ v\in X\,,
\end{equation*}%
and let $J$ from $X^{\prime }$ into $X$ be the inverse of $L.$ Then
\begin{equation*}
(J\,u,v)=\left\langle u,v\right\rangle \quad \text{for\ \ }u\in X^{\prime
}\,,\ v\in X\,.
\end{equation*}

Let $E$ be a $C^{1}$ functional defined on $X.$ We say that $E$ satisfies
the \textit{modified Palais-Smale-Schechter condition} (MPSS) in $%
K_{R_{0}R_{1}},$ if any sequence $\left( u_{k}\right) $ of elements of $%
K_{R_{0}R_{1}}$ for which $\left( E\left( u_{k}\right) \right) $ converges
and one of the following conditions holds:

(i) \ $E^{\prime }\left( u_{k}\right) \rightarrow 0;$

(ii) \ $\left\Vert u_{k}\right\Vert =R_{0},\ \ \left( JE^{\prime }\left(
u_{k}\right) ,Ju_{k}\right) \geq 0$ \ and\ \ $JE^{\prime }\left(
u_{k}\right) -\frac{\left( JE^{\prime }\left( u_{k}\right) ,Ju_{k}\right) }{%
\left\vert Ju_{k}\right\vert ^{2}}Ju_{k}\rightarrow 0;$

(iii)\ \ $\left\vert u_{k}\right\vert =R_{1},\ \ \left( JE^{\prime }\left(
u_{k}\right) ,u_{k}\right) \leq 0\ \ $and\ \ $JE^{\prime }\left(
u_{k}\right) -\frac{\left( JE^{\prime }\left( u_{k}\right) ,u_{k}\right) }{%
R_{1}^{2}}u_{k}\rightarrow 0,$

\noindent has a convergent subsequence.

We say that $E$ satisfies the \textit{compression boundary condition} in $%
K_{R_{0}R_{1}}$ if%
\begin{equation}
JE^{\prime }\left( u\right) -\lambda Ju\neq 0\ \ \text{for }u\in
K_{R_{0}R_{1}},\ \left\Vert u\right\Vert =R_{0},\ \lambda >0;  \label{h1}
\end{equation}%
\begin{equation}
JE^{\prime }\left( u\right) +\lambda u\neq 0\ \ \text{for }u\in
K_{R_{0}R_{1}},\ \left\vert u\right\vert =R_{1},\ \lambda >0.  \label{h2}
\end{equation}

We say that $E$ has a \textit{mountain pass geometry} in $K_{R_{0}R_{1}}$ if
there exist $u_{0}$ and $u_{1}$ in the same connected component of $%
K_{R_{0}R_{1}}, $ and $r>0$ such that $\left\vert u_{0}\right\vert
<r<\left\vert u_{1}\right\vert $ and%
\begin{equation*}
\max \left\{ E\left( u_{0}\right) ,E\left( u_{1}\right) \right\} <\inf
\left\{ E\left( u\right) :\ u\in K_{R_{0}R_{1}},\ \left\vert u\right\vert
=r\right\} .
\end{equation*}%
In this case we consider the set
\begin{equation}  \label{e-gamma}
\Gamma =\left\{ \gamma \in C([0,1];K_{R_{0}R_{1}})\ :\ \gamma (0)=u_{0},\
\gamma (1)=u_{1}\right\} \,
\end{equation}%
and the number%
\begin{equation}  \label{e-c}
c=\inf_{\gamma \in \Gamma }\,\max_{t\in \lbrack 0,1]}E(\gamma (t))\,.
\end{equation}

Finally, we say that $E$ is \textit{bounded from below} in $K_{R_{0}R_{1}}$
if
\begin{equation}  \label{e-m}
m:=\inf_{u\in K_{R_{0}R_{1}}}E\left( u\right) >-\infty .
\end{equation}

We assume that the following conditions are satisfied:

\begin{equation}
\left( I-JE^{\prime }\right) \left( K\right) \subset K\ \ \ (I\ \text{is the
identity map on\ }X);  \label{c1}
\end{equation}%
and there exists a constant $\nu _{0}>0$ such that
\begin{equation}
\left( JE^{\prime }\left( u\right) ,Ju\right) \leq \nu _{0}\ \ \text{for
all\ \ }u\in K\ \ \text{with\ \ }\left\Vert u\right\Vert =R_{0};  \label{b1}
\end{equation}%
\begin{equation}
\left( JE^{\prime }\left( u\right) ,u\right) \geq -\nu _{0}\ \ \text{for
all\ \ }u\in K\ \ \text{with\ \ }\left\vert u\right\vert =R_{1}.  \label{b2}
\end{equation}
The following theorems of localization of critical points in a conical shell
appear as slight particularizations of the main results from \cite{Pre, Pre2}%
.

\begin{theorem}
\label{t1} Assume that $E$ is bounded from bellow in $K_{R_{0}R_{1}}$ and
that there is a $\rho >0$ with
\begin{equation}
E\left( u\right) \geq m+\rho  \label{r11}
\end{equation}%
\emph{(}$m$ given in \emph{\eqref{e-m})} for all $u\in K_{R_{0}R_{1}}$ which
simultaneously satisfy $\left\vert u\right\vert =R_{1},\ \left\Vert
u\right\Vert =R_{0}.$ In addition assume that $E$ satisfies the (MPSS)
condition and the compression boundary condition in $K_{R_{0}R_{1}}.$ Then
there exists $u\in K_{R_{0}R_{1}}$ such that
\begin{equation*}
E^{\prime }\left( u\right) =0\ \ \ \text{and\ \ \ }E\left( u\right) =m.
\end{equation*}
\end{theorem}

\begin{theorem}
\label{t2}Assume that $E$ has the mountain pass geometry in $K_{R_{0}R_{1}}$
and that there is a $\rho >0$ with%
\begin{equation}
\left\vert E\left( u\right) -c\right\vert \geq \rho  \label{r12}
\end{equation}%
\emph{(}$c$ given in \emph{\eqref{e-c})} for all $u\in K_{R_{0}R_{1}}$ which
simultaneously satisfy $\left\vert u\right\vert =R_{1},\ \left\Vert
u\right\Vert =R_{0}.$ In addition assume that $E$ satisfies the (MPSS)
condition and the compression boundary condition in $K_{R_{0}R_{1}}.$ Then
there exists $u\in K_{R_{0}R_{1}}$ such that
\begin{equation*}
E^{\prime }\left( u\right) =0\ \ \ \text{and\ \ \ }E\left( u\right) =c.
\end{equation*}
\end{theorem}

\begin{remark}
\label{r} If the assumptions of both Theorems \ref{t1}, \ref{t2} are
satisfied, since $m<c,$ then $E$ has two distinct critical points in $%
K_{R_{0}R_{1}}.$
\end{remark}

\subsection{Krasnoselskii's compression-expansion theorem}

The problem (\ref{p}) can also be investigated by means of fixed point
techniques. In this paper, we are mainly concerned with the variational
approach based on critical point theory. However, it deserves to comment
about the applicability of fixed point methods and the surplus of
information given by the variational approach.

Thus we shall report on the applicability of Krasnoselskii's
compression-expansion theorem (see \cite{GL, KZ}), which guarantees the
existence of a fixed point of a compact operator in a conical shell of a
Banach space.

\begin{theorem}[Krasnoselskii]
\label{tk}Let $\left( X,\left\vert \cdot \right\vert \right) $ be a Banach
space and $K\subset X$ a cone. Let $R_{0},R_{1}$ be two numbers with $%
0<R_{0}<R_{1},$ $K_{R_{0}R_{1}}=\{u\in K:\ R_{0}\leq \left\vert u\right\vert
\leq R_{1}\},$ and let $N:K_{R_{0}R_{1}}\rightarrow K$ be a compact
operator. Let $<$ be the strict ordering induced in $X$ by the cone $K,$
i.e. $u<v$ if and only if $v-u\in K\setminus \left\{ 0\right\} .$ Assume
that one of the following conditions is satisfied:

\begin{description}
\item[(a)] compression: \emph{(i)} $N\left( u\right) \nless u$ for all $u\in
K$ with $\left\vert u\right\vert =R_{0},$ and \emph{(ii)} $N\left( u\right)
\ngtr u$ for all $u\in K$ with $\left\vert u\right\vert =R_{1};$

\item[(b)] expansion: \emph{(i)} $N\left( u\right) \ngtr u$ for all $u\in K$
with $\left\vert u\right\vert =R_{0},$ and \emph{(ii)} $N\left( u\right)
\nless u$ for all $u\in K$ with $\left\vert u\right\vert =R_{1}.$
\end{description}

Then $N$ has at least one fixed point in $K_{R_{0}R_{1}}.$
\end{theorem}

\subsection{Variational formulation of the problem}

Next we are going to describe the variational structure of the problem (\ref%
{p}) (see \cite{TC, YCY}).

Let $X$ be the Hilbert space
\begin{equation*}
X:=\{u\in H^{2}(0,1):\ u(0)=u^{\prime }(0)=0\}
\end{equation*}%
with inner product and norm
\begin{equation*}
(u,v):=\int_{0}^{1}u^{\prime \prime }(t)v^{\prime \prime }(t)dt,
\end{equation*}%
\begin{equation}
|u|:=\left( \int_{0}^{1}(u^{\prime \prime }(t))^{2}dt\right) ^{\frac{1}{2}}.
\label{en}
\end{equation}

We associate to the problem (\ref{p}), the functional $E:X\rightarrow
\mathbb{R}$ defined by
\begin{equation*}
E(u):=\frac{1}{2}\left\vert u\right\vert ^{2}-\,\int_{0}^{1}F\left(t,u\left(
t\right) \right) dt,
\end{equation*}%
where
\begin{equation*}
F(t,u)=\int_{0}^{u}f(t,s)ds\,.
\end{equation*}

The functional $E:X\rightarrow \mathbb{R}$ is $C^{1}$ and for any $u,v\in X,$
\begin{equation*}
<E^{\prime }(u),v>=\int_{0}^{1}\left( u^{\prime \prime }(t)\,v^{\prime
\prime }(t)-f(t,u(t))\,v(t)\right) \,dt.
\end{equation*}

Also, $u\in X$ is a critical point of $E$ if and only if $u$ is a classical
solution of the problem (\ref{p}) (see \cite{TC}).

In this specific case, $H=L^{2}\left( 0,1\right) $ with the usual norm
denoted by $\left\Vert \cdot \right\Vert ,$ and $L:X\rightarrow X^{\prime }$
is given by $Lu=u^{\left( 4\right) }\ $(in the distributional sense). The
inverse of $L$ is the operator $J\colon X^{\prime }\rightarrow X$ which at
each $v\in X^{\prime }$ attaches the unique $u\in X$ with $u^{\left(
4\right) }=v$ in the sense of distributions.

In particular, if $v\in L^{2}\left( 0,1\right) ,$ one has the representation
\begin{equation}
\left( Jv\right) \left( t\right) =\,\int_{0}^{1}G(t,s)\,v\left( s\right)
\,ds\,,  \notag
\end{equation}%
where $G(t,s)$ is the Green's function related to the fourth order problem
\begin{equation*}
\left\{
\begin{array}{ll}
u^{(4)}(t)=v(t), & 0<t<1, \\
u(0)=u^{\prime }(0)=u^{\prime \prime }(1)=u^{\prime \prime \prime }(1)=0. &
\end{array}%
\right.
\end{equation*}%
By means of the Mathematica package developed in \cite{cacima}, we have that
such function is given by the following expression:
\begin{equation}
G(t,s)=\left\{
\begin{array}{lc}
\dfrac{s^{2}}{6}\,(3\,t-s)\,, & 0\leq s\leq t\leq 1, \\
&  \\
\dfrac{t^{2}}{6}\,(3\,s-t)\,, & 0\leq t<s\leq 1\,.%
\end{array}%
\right.  \label{e-green}
\end{equation}%
Then the problem (\ref{p}) is equivalent to the integral equation
\begin{equation}
u(t)=\int_{0}^{1}G(t,s)\,f(s,u(s))\,ds\,,\ \ u\in C\left[ 0,1\right] .
\label{fpp}
\end{equation}%
We note that $E^{\prime }\left( u\right) =Lu-f\left( \cdot ,u\right) $ and $%
JE^{\prime }\left( u\right) =u-N\left( u\right) ,$ where
\begin{equation*}
N\left( u\right) =Jf\left( \cdot ,u\right) =\int_{0}^{1}G(\cdot
,s)\,f(s,u(s))\,ds.
\end{equation*}%
Obviously, (\ref{fpp}) represents a fixed point equation associated to $N.$
Note also that, since the embedding $X\subset C([0,1])$ is compact and $f$
is a continuous function, $N$ is a compact operator from $X$ to $X.$

\section{Main results}

\subsection{Localization in a shell defined by the energetic norm}

First we shall deal with the localization of positive solutions $u$ of the
problem (\ref{p}) in a shell defined by a single norm, more exactly

\begin{equation*}
R_{0}\leq \left\vert u\right\vert \leq R_{1},
\end{equation*}%
where $\left\vert \cdot \right\vert $ is the energetic norm given by (\ref%
{en}). For this, the following unilateral Harnack inequality is crucial.

\begin{lemma}
\label{LH}If $u\in C^{4}\left[ 0,1\right] $ satisfies $u(0)=u^{\prime
}(0)=u^{\prime \prime }(1)=u^{\prime \prime \prime }(1)=0$ and $u^{\left(
4\right) }$ is nonnegative and nondecreasing in $\left[ 0,1\right] ,$ then $%
u $ is convex and%
\begin{equation}
u\left( t\right) \geq M_{0}\left( t\right) \left\vert u\right\vert \ \ \
\text{for all }t\in \left[ 0,1\right] ,  \label{hi}
\end{equation}%
where \ $M_{0}\left( t\right) =\frac{\sqrt{2}\left( 1-t\right) t^{3}}{6}.$
\end{lemma}

\begin{proof}
From $u^{\left( 4\right) }\geq 0$ it follows that $u^{\prime \prime }$ is
convex. This together with $u^{\prime \prime }\left( 1\right) =\left(
u^{\prime \prime }\right) ^{\prime }\left( 1\right) =0$ gives that $%
u^{\prime \prime }$ is nonnegative and nonincreasing. Next, from $u^{\prime
\prime }\geq 0$ one has that $u$ is convex, and since $u\left( 0\right)
=u^{\prime }\left( 0\right) =0,$ $u$ must be nondecreasing and nonnegative.

On the other hand, since $u^{\left( 4\right) }\geq 0$ we have that $%
u^{\prime \prime \prime }$ is nondecreasing and since $u^{\prime \prime
\prime }\left( 1\right) =0,$ $u^{\prime \prime \prime }\leq 0.$ Then $%
u^{\prime }$ is concave; it is also nondecreasing due to $u^{\prime \prime
}\geq 0,$ and since $u^{\prime }\left( 0\right) =0,$ we have $u^{\prime
}\geq 0.$ Now from $u^{\prime \prime }\geq 0,$ $u^{\prime }\geq 0$ and $%
u\left( 0\right) =0,$ we see that $u$ is nonnegative, nondecreasing and
convex.

Finally note that from $u^{\left( 4\right) }$ nondecreasing, we have that $%
u^{\prime \prime \prime }$ is convex, and since $u^{\prime \prime \prime
}\left( 1\right) =0,$ the graph of $u^{\prime \prime \prime }$ is under the
line connecting the points $\left( 0,u^{\prime \prime \prime }\left(
0\right) \right) $ and $\left( 1,0\right) ,$ i.e.%
\begin{equation}
u^{\prime \prime \prime }\left( t\right) \leq \left( 1-t\right) u^{\prime
\prime \prime }\left( 0\right) ,\ \ \ t\in \left[ 0,1\right] .  \label{ci}
\end{equation}

Due to the fact that the function $u^{\prime \prime }$ is nonincreasing and
the function $u^{\prime \prime \prime}$ is nondecreasing we have:
\begin{eqnarray*}
u\left( t\right) &=&\int_{0}^{t}\int_{0}^{s}u^{\prime \prime }\left( \tau
\right) d\tau ds\geq \int_{0}^{t}\int_{0}^{s}u^{\prime \prime }\left(
s\right) d\tau ds=\int_{0}^{t}su^{\prime \prime }\left( s\right) ds \\
&=&\frac{t^{2}}{2}u^{\prime \prime }\left( t\right) -\int_{0}^{t}\frac{s^{2}%
}{2}u^{\prime \prime \prime }\left( s\right) ds\geq -\int_{0}^{t}\frac{s^{2}%
}{2}u^{\prime \prime \prime }\left( s\right) ds \\
&\geq &-\int_{0}^{t}\frac{s^{2}}{2}u^{\prime \prime \prime }\left( t\right)
ds=-\frac{t^{3}}{6}u^{\prime \prime \prime }\left( t\right) .
\end{eqnarray*}%
This inequality combined with (\ref{ci}) gives%
\begin{equation}
u\left( t\right) \geq -\frac{\left( 1-t\right) t^{3}}{6}u^{\prime \prime
\prime }\left( 0\right) .  \label{e0}
\end{equation}%
Next we deal with the energetic norm wishing to connect it to $u^{\prime
\prime \prime }\left( 0\right) .$ One has%
\begin{eqnarray}
\left\vert u\right\vert ^{2} &=&\int_{0}^{1}u^{\prime \prime }\left(
t\right) ^{2}dt=\left. u^{\prime \prime }u^{\prime }\right\vert
_{0}^{1}-\int_{0}^{1}u^{\prime \prime \prime }\left( t\right) u^{\prime
}\left( t\right) dt  \label{e1} \\
&=&-\int_{0}^{1}u^{\prime \prime \prime }\left( t\right) u^{\prime }\left(
t\right) dt\leq -u^{\prime \prime \prime }\left( 0\right) u^{\prime }\left(
1\right) .  \notag
\end{eqnarray}%
Also%
\begin{eqnarray}
u^{\prime }\left( 1\right) &=&\int_{0}^{1}u^{\prime \prime }\left( t\right)
dt=-\int_{0}^{1}\int_{t}^{1}u^{\prime \prime \prime }\left( s\right)
dsdt\leq -\int_{0}^{1}\int_{t}^{1}u^{\prime \prime \prime }\left( t\right)
dsdt  \label{e2} \\
&=&-\int_{0}^{1}\left( 1-t\right) u^{\prime \prime \prime }\left( t\right)
dt\leq -\int_{0}^{1}\left( 1-t\right) u^{\prime \prime \prime }\left(
0\right) dt=-\frac{1}{2}u^{\prime \prime \prime }\left( 0\right) .  \notag
\end{eqnarray}%
From (\ref{e1}) and (\ref{e2}) we deduce $\ \left\vert u\right\vert ^{2}\leq
\frac{1}{2}u^{\prime \prime \prime }\left( 0\right) ^{2},$ or%
\begin{equation*}
-u^{\prime \prime \prime }\left( 0\right) \geq \sqrt{2}\left\vert
u\right\vert .
\end{equation*}%
This inequality and (\ref{e0}) prove (\ref{hi}).
\end{proof}

Consider the cone%
\begin{equation*}
K:=\left\{ u\in X:u\text{ convex and}\ u\left( t\right) \geq M_{0}\left(
t\right) \left\vert u\right\vert \text{ \ on }\left[ 0,1\right] \right\} .
\end{equation*}

We note that, since any convex function with $u\left( 0\right) =u^{\prime
}\left( 0\right) =0$ is nondecreasing, all the elements of $K$ are
nondecreasing functions.

Also $K\neq \left\{ 0\right\} .$ Indeed, if we consider the eigenvalue
problem
\begin{equation*}
\left\{
\begin{array}{l}
\phi ^{(4)}(t)=\lambda \,\phi (t)\,,\quad 0<t<1\,, \\
\phi (0)=\phi ^{\prime }(0)=\phi ^{\prime \prime }(1)=\phi ^{\prime \prime
\prime }(1)=0\,,%
\end{array}%
\right.
\end{equation*}%
then its first eigenvalue $\lambda _{1}=\beta ^{4},$ where $\beta =\dfrac{%
\pi }{2}+0.3042$ is the smallest positive solution of the equation
\begin{equation*}
\cos \lambda \cosh \lambda +1=0\,,
\end{equation*}%
while the function
\begin{equation}
\phi _{1}(t)=\sin \beta \,t-\sinh \beta \,t+\dfrac{\sinh \beta +\sin \beta }{%
\cosh \beta +\cos \beta }\,(\cosh \beta \,t-\cos \beta \,t)\,  \label{ef}
\end{equation}%
is a positive eigenfunction (see \cite{S}) corresponding to $\lambda _{1}.$
In addition, one can check that $\phi _{1}^{\prime \prime }\geq 0,$ that is $%
\phi _{1}$ is convex. Since $\phi _{1}\left( 0\right) =\phi _{1}^{\prime
}\left( 0\right) =0,$ we must have $\phi _{1}\geq 0$ and $\phi _{1}^{\prime
}\geq 0$ on $\left[ 0,1\right] $. As consequence, $\phi _{1}^{(5)}=\lambda
\,\phi _{1}^{\prime }\geq 0$ on $[0,1].$ Then, according to Lemma \ref{LH}, $%
\phi _{1}\in K.$

Denote
\begin{equation*}
M_{1}\left( t\right) :=\frac{2}{3}t^{\frac{3}{2}},\ \ \ \ t\in \left[ 0,1%
\right] .\ \ \
\end{equation*}%
Our assumptions on $f$ are as follows:

\begin{description}
\item[(h1)] $f$ is nondecreasing on $\left[ 0,1\right] \times \mathbb{R}_{+}$
in each of its variables;

\item[(h2)] there exist $R_{0},R_{1}$ with $0<R_{0}<R_{1}$ such that

\begin{description}
\item[(a)] $\int_{0}^{1}M_{0}\left( t\right) f\left( t,M_{0}\left( t\right)
R_{0}\right) dt\geq R_{0},$

\item[(b)] $\int_{0}^{1}M_{1}\left( t\right) f\left( t,M_{1}\left( t\right)
R_{1}\right) dt\leq R_{1}.$
\end{description}

\item[(h3)] there exist $u_{0},u_{1}\in K_{R_{0}R_{1}}=\{u\in K:R_{0}\leq
\left\vert u\right\vert \leq R_{1}\}$ and $r>0$ such that $\left\vert
u_{0}\right\vert <r<\left\vert u_{1}\right\vert $ and%
\begin{equation*}
\max \left\{ E\left( u_{0}\right) ,\ E\left( u_{1}\right) \right\} <\inf
\left\{ E\left( u\right) :\ u\in K,\ \left\vert u\right\vert =r\right\} .
\end{equation*}
\end{description}

\begin{theorem}
\label{t3}Assume that \emph{(h1), (h2)} are satisfied. Let $\Gamma ,$ $m$
and $c$ be defined as in \emph{\eqref{e-gamma}, \eqref{e-c}} and \emph{%
\eqref{e-m}} respectively. Then the fourth-order problem \emph{(\ref{p})}
has at least one positive solution $u_{m}$ in $K_{R_{0}R_{1}}$ such that
\begin{equation*}
E\left( u_{m}\right) =m.
\end{equation*}%
If in addition \emph{(h3)} holds, then a second positive solution $u_{c}$
exists in $K_{R_{0}R_{1}}$ with
\begin{equation*}
E\left( u_{c}\right) =c.
\end{equation*}
\end{theorem}

\begin{proof}
First let us note that the (MPSS) condition holds in $K_{R_{0}R_{1}}$ due to
the compactness of the operator $N=I-JE^{\prime }.$ Also the boundedness of $%
\left( JE^{\prime }\left( u\right) ,Ju\right) $ and $\left( JE^{\prime
}\left( u\right) ,u\right) $ on the boundaries of $K_{R_{0}R_{1}},$ i.e. (%
\ref{b1}) and (\ref{b2}) is guaranteed since $JE^{\prime }$ maps bounded
sets into bounded sets.

To check (\ref{c1}), let $u$ be any element of $K.$ Hence $u$ is nonnegative
and nondecreasing on $\left[ 0,1\right] .$ Then, from (h1) we also have that
$f\left(t,u\left(t\right)\right)$ is nonnegative and nondecreasing in $\left[
0,1\right] .$ Now, Lemma \ref{LH} implies that $Jf\left( \cdot ,u\left(
\cdot \right) \right) \in K.$ But $Jf\left( \cdot ,u\left( \cdot \right)
\right) =\left( I-JE^{\prime }\right) \left( u\right) .$ Thus (\ref{c1})
holds.

Next, let us note that for any $u\in K_{R_{0}R_{1}},$ we have
\begin{eqnarray}
u\left( t\right) &=&\int_{0}^{t}\int_{0}^{s}u^{\prime \prime }\left( \tau
\right) d\tau ds\leq \int_{0}^{t}\sqrt{s}\left( \int_{0}^{s}u^{\prime \prime
}\left( \tau \right) ^{2}d\tau \right) ^{1/2}ds  \label{r14} \\
&\leq &\left\vert u\right\vert \int_{0}^{t}\sqrt{s}ds=\frac{2}{3}t^{\frac{3}{%
2}}\left\vert u\right\vert =M_{1}\left( t\right) \left\vert u\right\vert .
\notag
\end{eqnarray}%
Then%
\begin{equation*}
E\left( u\right) =\frac{1}{2}\left\vert u\right\vert
^{2}-\int_{0}^{1}F\left( t,u\left( t\right) \right) dt\geq \frac{1}{2}%
R_{0}^{2}-F\left( 1,\frac{2}{3}R_{1}\right) .
\end{equation*}%
Hence $E$ is bounded from below on $K_{R_{0}R_{1}}.$

Furthermore, we check the boundary conditions (\ref{h1}). Assume that $%
JE^{\prime }\left( u\right) -\lambda Ju=0$ for some $u\in K$ with $%
\left\vert u\right\vert =R_{0}$ and $\lambda >0.$ Then $\ u$ solves the
problem%
\begin{equation*}
\left\{
\begin{array}{ll}
u^{(4)}(t)-\,f(t,u(t))-\lambda u\left( t\right) =0, & 0<t<1, \\
u(0)=u^{\prime }(0)=u^{\prime \prime }(1)=u^{\prime \prime \prime }(1)=0. &
\end{array}%
\right\}
\end{equation*}

and
\begin{eqnarray*}
R_{0}^{2} &=&\left\vert u\right\vert ^{2}=\int_{0}^{1}\left[ f\left(
t,u\left( t\right) \right) +\lambda u\left( t\right) \right] u\left(
t\right) dt \\
&\geq &\int_{0}^{1}\left[ f\left( t,M_{0}\left( t\right) R_{0}\right)
+\lambda M_{0}\left( t\right) R_{0}\right] M_{0}\left( t\right) R_{0}dt \\
&>&R_{0}\,\int_{0}^{1}f\left( t,M_{0}\left( t\right) R_{0}\right)
M_{0}\left( t\right)\,dt,
\end{eqnarray*}%
which contradicts the assumption (h2) (a). Hence $JE^{\prime }\left(
u\right) -\lambda Ju\neq 0$ for all $u\in K$ with $\left\vert u\right\vert
=R_{0}$ and $\lambda >0.$

Assume now that $JE^{\prime }\left( u\right) +\lambda u=0$ for some $u\in K$
with $\left\vert u\right\vert =R_{1}$ and $\lambda >0.$ Then $u$ solves the
problem
\begin{equation*}
\left\{
\begin{array}{ll}
\left( 1+\lambda \right) u^{(4)}(t)-\,f(t,u(t))=0, & 0<t<1, \\
u(0)=u^{\prime }(0)=u^{\prime \prime }(1)=u^{\prime \prime \prime }(1)=0. &
\end{array}%
\right.
\end{equation*}%
Then%
\begin{equation*}
R_{1}^{2}=\left\vert u\right\vert ^{2}=\frac{1}{1+\lambda }%
\int_{0}^{1}f\left( t,u\left( t\right) \right) u\left( t\right) dt.
\end{equation*}%
Using (\ref{r14}) we deduce%
\begin{equation*}
R_{1}^{2}<\int_{0}^{1}M_{1}\left( t\right) R_{1}f\left( t,M_{1}\left(
t\right) R_{1}\right) dt,
\end{equation*}%
which contradicts (h2) (b). Hence $JE^{\prime }\left( u\right) +\lambda
u\neq 0$ for all $u\in K$ with $\left\vert u\right\vert =R_{1}$ and $\lambda
>0.$

The conclusions follow from Theorem \ref{t1} and Theorem \ref{t2}.
\end{proof}

For the autonomous case $f\left( t,u\right) =f\left( u\right) ,$ where $f$
is nonnegative and nondecreasing on $\mathbb{R}_{+},$ we may replace the
conditions of (h2) by a couple of simpler inequalities.

\begin{example}
We give an example of a function $f(u)$ which satisfies the conditions $%
(h_{1})$ and $(h_{2})$ of Theorem 3.2. Note that $\ $%
\begin{equation*}
0\leq \frac{\sqrt{2}}{6}(1-t)t^{3}<0.03\ \ \ \ \text{if\ \ \ }0\leq t\leq 1,
\end{equation*}%
\begin{equation*}
\int_{0}^{1}\left( \frac{\sqrt{2}}{6}(1-t)t^{3}\right) ^{2}dt=\frac{1}{4536},
\end{equation*}%
and $4600\times \frac{3}{100}=138.$ Define
\begin{equation*}
f\left( u\right) =\left\{
\begin{array}{ll}
0, & u\leq 0, \\
4600\,u, & 0\leq u\leq 0.03, \\
138, & u\geq 0.03.%
\end{array}%
\right.
\end{equation*}%
Taking $R_{0}=1$ and $R_{1}=37,$ by
\begin{equation*}
\int_{0}^{1}138\,\frac{2}{3}t^{\frac{3}{2}}dt=\frac{184}{5}=36.8,
\end{equation*}%
we obtain that the conditions $(h_{1})$ and $(h_{2})$ are satisfied.
\end{example}

Further, we have:

\begin{theorem}
\label{t4}Assume that $f:\mathbb{R}_{+}\rightarrow \mathbb{R}_{+}$ is
continuous nondecreasing and that for some numbers $a\in \left( 0,1\right) ,$
$R_{0}$ and $R_{1}$ with $0<R_{0}<R_{1},$ one has
\begin{equation}
\frac{f\left( M_{0}\left( a\right) R_{0}\right) }{M_{0}\left( a\right) R_{0}}%
\geq \frac{1}{\left( 1-a\right) M_{0}\left( a\right) ^{2}},\ \ \ \frac{%
f\left( \frac{2}{3}R_{1}\right) }{R_{1}}\leq \frac{15}{4}.  \label{f2}
\end{equation}%
Then \emph{(\ref{p})} has at least one positive solution $u_{m}$ in $%
K_{R_{0}R_{1}}$ with $E\left( u_{m}\right) =m.$ If in addition \emph{(h3)}
holds, then a second positive solution $u_{c}$ exists in $K_{R_{0}R_{1}}$
with $E\left( u_{c}\right) =c.$
\end{theorem}

\begin{proof}
Since $M_{1}\left( t\right) \leq 2/3$ for every $t\in \left[ 0,1\right] ,$
we have%
\begin{equation*}
\int_{0}^{1}M_{1}\left( t\right) f\left( M_{1}\left( t\right) R_{1}\right)
dt\leq f\left( \frac{2}{3}R_{1}\right) \int_{0}^{1}\frac{2}{3}t^{\frac{3}{2}%
}dt=\frac{4}{15}f\left( \frac{2}{3}R_{1}\right) .
\end{equation*}%
Then the inequality
\begin{equation*}
\frac{4}{15}f\left( \frac{2}{3}R_{1}\right) \leq R_{1},\text{ }
\end{equation*}%
or equivalently the second inequality in (\ref{f2}) is a sufficient
condition for (h2)(b) to hold. As concerns the first inequality in (\ref{f2}%
), let us remark that if $JE^{\prime }\left( u\right) -\lambda Ju=0$ for
some $u\in K$ with $\left\vert u\right\vert =R_{0}$ and $\lambda >0,$ then%
\begin{eqnarray}
R_{0}^{2} &=&\left\vert u\right\vert ^{2}=\int_{0}^{1}\left[ f\left( u\left(
t\right) \right) +\lambda u\left( t\right) \right] u\left( t\right) dt
\label{f3} \\
&\geq &\int_{a}^{1}\left[ f\left( u\left( t\right) \right) +\lambda u\left(
t\right) \right] u\left( t\right) dt.  \notag
\end{eqnarray}%
The function $u$ being nondecreasing, one has $u\left( t\right) \geq u\left(
a\right) $ for all $t\in \left[ a,1\right] .$ Also, from (\ref{hi}), $%
u\left( a\right) \geq M_{0}\left( a\right) \left\vert u\right\vert .$ Then
from (\ref{f3}),
\begin{eqnarray*}
R_{0}^{2} &\geq &\left( 1-a\right) \left[ f\left( M_{0}\left( a\right)
R_{0}\right) +\lambda M_{0}\left( a\right) R_{0}\right] M_{0}\left( a\right)
R_{0} \\
&>&\left( 1-a\right) f\left( M_{0}\left( a\right) R_{0}\right) M_{0}\left(
a\right) R_{0}.
\end{eqnarray*}%
Hence
\begin{equation*}
R_{0}>\left( 1-a\right) M_{0}\left( a\right) f\left( M_{0}\left( a\right)
R_{0}\right) ,
\end{equation*}%
i.e.~the opposite of the first inequality in (\ref{f2}).
\end{proof}

Clearly the inequalities (\ref{f2}) express the oscillation of the function $%
f\left( t\right) /t$ up and down the values $1/\left( 1-a\right) M_{0}\left(
a\right) ^{2}$ and $45/8.$

\begin{remark}[Existence asymptotic conditions]
The existence of two numbers $R_{0},$ $R_{1}$ satisfying (\ref{f2}) is
guaranteed by the asymptotic conditions%
\begin{equation*}
\lim \sup_{\tau \rightarrow 0}\frac{f\left( \tau \right) }{\tau }>\frac{1}{%
\left( 1-a\right) M_{0}\left( a\right) ^{2}}\ \ \ \text{and\ \ \ }\lim
\inf_{\tau \rightarrow \infty }\frac{f\left( \tau \right) }{\tau }<\frac{45}{%
8}.
\end{equation*}
\end{remark}

\begin{remark}[Multiplicity]
Theorems \ref{t3} and \ref{t4} can be used to obtain multiple positive
solutions. Indeed, if their assumptions are fulfilled for two pairs $\left(
R_{0},R_{1}\right) ,$ $\left( \overline{R}_{0},\overline{R}_{1}\right) ,$
then we obtain four solutions, provided that the sets $K_{R_{0}R_{1}}$ and $%
K_{\overline{R}_{0}\overline{R}_{1}}$ are disjoint. This happens if $%
0<R_{0}<R_{1}<\overline{R}_{0}<\overline{R}_{1}.$ We can even obtain
sequences of positive solutions; for instance, in connection with Theorem %
\ref{t4}, if
\begin{equation*}
\lim \sup_{\tau \rightarrow 0}\frac{f\left( \tau \right) }{\tau }>\frac{1}{%
\left( 1-a\right) M_{0}\left( a\right) ^{2}}\ \ \ \text{and\ \ \ }\lim
\inf_{\tau \rightarrow 0}\frac{f\left( \tau \right) }{\tau }<\frac{45}{8},
\end{equation*}%
then there exists a sequence $\left( u_{k}\right) $ of positive solutions
with $u_{k}\rightarrow 0$ as $k\rightarrow \infty .$ Also, if
\begin{equation*}
\lim \sup_{\tau \rightarrow \infty }\frac{f\left( \tau \right) }{\tau }>%
\frac{1}{\left( 1-a\right) M_{0}\left( a\right) ^{2}}\ \ \ \text{and\ \ \ }%
\lim \inf_{\tau \rightarrow \infty }\frac{f\left( \tau \right) }{\tau }<%
\frac{45}{8},
\end{equation*}%
then there exists a sequence $\left( u_{k}\right) $ of positive solutions
with $\left\vert u_{k}\right\vert \rightarrow \infty $ as $k\rightarrow
\infty .$
\end{remark}

\begin{remark}[Fixed point approach]
Under the assumptions of Theorem \ref{t3}, the existence of a solution in $%
K_{R_{0}R_{1}}$ can also be obtained via Krasnoselskii's theorem. Indeed,
the problem (\ref{p}) is equivalent to the fixed point problem (\ref{fpp})
in $X$ for the compact operator $N:K_{R_{0}R_{1}}\rightarrow K,$ $N\left(
u\right) =Jf\left( \cdot ,u\left( \cdot \right) \right).$

Let us check the condition (a)(i). Assume the contrary, i.e. $Nu<u$ for some
$u\in K$ with $\left\vert u\right\vert =R_{0}.$ Then $Nu=u-v$ for some $v\in
K\setminus \left\{ 0\right\} .$ This means that $\left( u-v\right) ^{\left(
4\right) }=f\left( t,u\right) $ in the sense of distributions. Now multiply
by $u$ and integrate to obtain%
\begin{equation*}
\left\vert u\right\vert ^{2}-\int_{0}^{1}u^{\prime \prime }\left( t\right)
v^{\prime \prime }\left( t\right) dt=\int_{0}^{1}f\left( t,u\left( t\right)
\right) u\left( t\right) dt.
\end{equation*}%
Since $v,u-v\in K,$ one has $v^{\prime \prime }\geq 0$ and $u^{\prime \prime
}-v^{\prime \prime }\geq 0$ in $\left[ 0,1\right] .$ Hence%
\begin{equation*}
\int_{0}^{1}u^{\prime \prime }\left( t\right) v^{\prime \prime }\left(
t\right) dt\geq \int_{0}^{1}v^{\prime \prime }\left( t\right)
^{2}dt=\left\vert v\right\vert ^{2}>0.
\end{equation*}%
Then%
\begin{equation*}
R_{0}^{2}=\left\vert u\right\vert ^{2}>\int_{0}^{1}f\left( t,u\left(
t\right) \right) u\left( t\right) dt.
\end{equation*}%
Next we use (\ref{hi}) to derive a contradiction to (h2) (a).

The condition (a)(ii) can be proved similarly.

Notice that under assumptions (h2)(a) and (b), a solution exists in $%
K_{R_{1}R_{0}}$ in case that $R_{1}<R_{0}.$ However this is not guaranteed
by the variational approach.

We may conclude that, compared to the fixed point approach, the variational
method gives an additional information about the solution, namely of being a
minimum for the energy functional. Moreover, a second solution of mountain
pass type can be guaranteed by the variational approach.
\end{remark}

The above approach was essentially based on the monotonicity assumption on $%
f,$ which was required by the Harnack type inequality (\ref{hi}). Thus a
natural question is if such an inequality can be established for functions $%
u $ satisfying the boundary conditions and $u^{\left( 4\right) }\geq 0,$
without the assumption that $u^{\left( 4\right) }$ is nondecreasing. In the
absence of the answer to this question, an alternative approach is possible
in a shell defined by two norms as shown in the next section.

\subsection{Localization in a shell defined by two norms}

In the previous section, a unilateral Harnack inequality was established for
functions $\ u$ satisfying the two point boundary conditions and with $%
u^{\left( 4\right) }$ nonnegative and nondecreasing in $\left[ 0,1\right] ,$
in terms of the energetic norm. If we renounce to the monotonicity of $%
u^{\left( 4\right) },$ then we have the following result in terms of the max
norm.

\begin{lemma}
\label{LH2}If $u\in C^{4}\left[ 0,1\right] $ satisfies $u(0)=u^{\prime
}(0)=u^{\prime \prime }(1)=u^{\prime \prime \prime }(1)=0$ and $u^{\left(
4\right) }\geq 0$ in $\left[ 0,1\right] ,$ then
\begin{equation}
u(t)\geq M\left( t\right) \,\left\Vert u\right\Vert _{\infty }\,\quad \text{%
for all \ }t\in \lbrack 0,1]\,,  \label{r15}
\end{equation}%
where \ $M\left( t\right) :=\dfrac{(3-t)\,t^2}{3}.$
\end{lemma}

\begin{proof}
Fix $t\in (0,1).$ Considering the expression of the Green's function given
in \eqref{e-green}, we have
\begin{eqnarray}
\lim_{s\rightarrow 0^{+}}\dfrac{G(t,s)}{s^{2}} &=&\lim_{s\rightarrow 0^{+}}%
\dfrac{3\,t-s}{6}=\dfrac{t}{2}>0\,,  \notag \\
\lim_{s\rightarrow 1^{-}}\dfrac{G(t,s)}{s^{2}} &=&\lim_{s\rightarrow 1^{-}}%
\dfrac{t^{2}}{6}\left( \dfrac{3}{s}-\dfrac{t}{s^{2}}\right) =\dfrac{t^{2}}{6}%
(3-t)>0\,.  \notag
\end{eqnarray}

Then, we can affirm that for every $t\in (0,1)$ fixed, the function $\frac{%
G(t,s)}{s^{2}}$ is continuous and positive on $[0,1]$ and it has strictly
positive maximum and minimum on $[0,1].$

Let
\begin{equation}
H(t,s)=\dfrac{G(t,s)}{s^{2}}=\left\{
\begin{array}{lc}
H_{1}(t,s)=\dfrac{3\,t-s}{6}\,, & 0\leq s\leq t\leq 1, \\
H_{2}(t,s)=\dfrac{t^{2}}{6}\left( \dfrac{3}{s}-\dfrac{t}{s^{2}}\right) \,, &
0\leq t<s\leq 1,%
\end{array}%
\right.
\end{equation}

We have
\begin{equation*}
\dfrac{\partial }{\partial s}H_{1}(t,s)=-\dfrac{1}{6}<0.
\end{equation*}

Then $H_{1}(t,s)$ is a decreasing function for $s\in \lbrack 0,t];$ hence we
have that
\begin{equation*}
H_{1}(t,s)\leq H_{1}(t,0)=\dfrac{t}{2}\ \ \ \text{for\ \ }0\leq s\leq t\leq
1.
\end{equation*}

Also, we obtain

\begin{equation*}
\dfrac{\partial }{\partial s}H_{2}(t,s)=\dfrac{t^{2}}{6}\left( -\dfrac{3}{%
s^{2}}+\dfrac{2\,t}{s^{3}}\right)
\end{equation*}%
and
\begin{equation*}
\dfrac{\partial ^{2}}{\partial s^{2}}H_{2}(t,s)=\dfrac{t^{2}}{s^{4}}\left(
s-t\right) \geq 0\,,\quad 0\leq t<s\leq 1\,,
\end{equation*}%
then
\begin{equation*}
\dfrac{\partial }{\partial s}H_{2}(t,s)\leq \dfrac{\partial }{\partial s}%
H_{2}(t,s)_{\mid s=1}=\dfrac{t^{2}}{6}(2\,t-3)<0\,,
\end{equation*}%
hence%
\begin{equation*}
H_{2}(t,s)\geq H_{2}(t,1)=\dfrac{t^{2}}{6}(3-t)\ \ \text{for\ \ }0\leq
t<s\leq 1.
\end{equation*}%
Combining the previous results we obtain the following estimations for the
Green function
\begin{equation}  \label{e-des-green}
\dfrac{t^{2}}{6}(3-t)\,s^{2}\leq G(t,s)\leq \dfrac{s^{2}}{2}\,\quad \text{%
for all }(t,s)\in \lbrack 0,1]\times \lbrack 0,1]\,.
\end{equation}

Let $u\in C^{4}\left[ 0,1\right] $ satisfy $u(0)=u^{\prime }(0)=u^{\prime
\prime }(1)=u^{\prime \prime \prime }(1)=0$ and $u^{\left( 4\right) }\geq 0$
in $\left[ 0,1\right] .$ Then
\begin{eqnarray}
u(t) &=&\int_{0}^{1}G(t,s)u^{\left( 4\right) }\left( s\right) \,ds\geq
\int_{0}^{1}\dfrac{(3-t)\,t^{2}\,s^{2}}{6}\,u^{\left( 4\right) }\left(
s\right) \,ds  \notag \\
&=&\dfrac{(3-t)\,t^2}{6}\int_{0}^{1}\,s^{2}u^{\left( 4\right) }\left(
s\right) \,ds\geq \dfrac{(3-t)\,t^2}{3}\int_{0}^{1}\,\left\{\max_{t\in
\lbrack 0,1]}G(t,s)\right\}\,u^{\left( 4\right) }\left( s\right) \,ds  \notag
\\
&\ge&\dfrac{(3-t)\,t^2}{3}\max_{t\in \lbrack
0,1]}\left\{\int_{0}^{1}\,G(t,s)u^{\left( 4\right) }\left( s\right) \,ds
\right\}= \dfrac{(3-t)\,t^2}{3}\left\Vert u\right\Vert _{\infty }.  \notag
\end{eqnarray}
\end{proof}

Notice that, since $\left\Vert u\right\Vert \leq \left\Vert u\right\Vert
_{\infty },$ the inequality (\ref{r15}) also gives%
\begin{equation*}
u(t)\geq M\left( t\right) \,\left\Vert u\right\Vert \,\quad \text{for all \ }%
t\in \lbrack 0,1].
\end{equation*}

Using Lemma \ref{LH2}, the existence of a positive solution can be
immediately obtained via Krasnoselskii's theorem.

\begin{theorem}
Assume that there exist positive numbers $\alpha ,\beta ,$ $\alpha \neq
\beta $ such that%
\begin{equation}
\alpha \leq \left( J\underline{f_\alpha}\right) \left( 1\right) \ \ \ \text{%
and\ \ \ }\beta \geq \left( J\overline{f_\beta}\right) \left( 1\right) ,
\label{r0}
\end{equation}%
where%
\begin{eqnarray*}
\underline{f_\alpha}\left( t\right) &=&\min \left\{ f\left( t,u\right) :\ \
M\left( t\right) \alpha \leq u\leq \alpha \right\} , \\
\overline{f_\beta}\left( t\right) &=&\max \left\{ f\left( t,u\right) :\ \
M\left( t\right) \beta \leq u\leq \beta \right\} .
\end{eqnarray*}%
Then the problem \emph{(\ref{p})} has at least one positive solution $u$
such that%
\begin{equation*}
R_{0}\leq \left\Vert u\right\Vert _{\infty }\leq R_{1},
\end{equation*}%
where $R_{0}=\min \left\{ \alpha ,\beta \right\} ,\ R_{1}=\max \left\{
\alpha ,\beta \right\} .$
\end{theorem}

\begin{proof}
The problem (\ref{p}) is equivalent to the fixed point problem $N\left(
u\right) =u$ in $C\left[ 0,1\right] ,$ where $N\left( u\right) =Jf\left(
\cdot,u\left( \cdot \right) \right).$

In the space $C\left[ 0,1\right]$ we consider the cone%
\begin{equation*}
K=\left\{ u\in C\left[ 0,1\right] :\ \text{ }u\left( 0\right) =0,\ u\left(
t\right) \geq M\left( t\right) \left\Vert u\right\Vert _{\infty }\text{ for
all }t\in \left[ 0,1\right] \right\} .
\end{equation*}%
According to inequalities \eqref{e-des-green}, it is not difficult to verify
that $N\left( K\right) \subset K.$ Also $N$ is a compact operator.

Now we show that the required boundary conditions from Krasnoselskii's
theorem are satisfied. Assume by contradiction that $Nu<u$ for some $u\in K$
with $\left\Vert u\right\Vert _{\infty }=\alpha .$ Then $Nu=u-v$ for some $%
v\in K\setminus \left\{ 0\right\} .$ Hence
\begin{equation}
u\left( t\right) -v\left( t\right) =Jf\left( t,u\left( t\right) \right) .
\label{r1}
\end{equation}%
We have
\begin{equation*}
M\left( t\right) \alpha \leq u\left( t\right) \leq \alpha \ \ \ \text{for
all }t\in \left[ 0,1\right] .
\end{equation*}%
Hence
\begin{equation*}
f\left( t,u\left( t\right) \right) \geq \underline{f_{\alpha }}\left(
t\right) .
\end{equation*}%
Since $J$ is a positive linear operator, it preserves ordering, so $Jf\left(
t,u\left( t\right) \right) \geq \left( J\underline{f_{\alpha }}\right)
\left( t\right) .$ \ Returning to (\ref{r1}), we deduce that
\begin{equation}
u\left( t\right) -v\left( t\right) \geq \left( J\underline{f_{\alpha }}%
\right) \left( t\right) .  \label{r2}
\end{equation}%
Since $v\left( t\right) \geq M\left( t\right) \left\Vert v\right\Vert
_{\infty }>0,$ for $t>0,$ (\ref{r2}) yields%
\begin{equation*}
\alpha =\left\Vert u\right\Vert _{\infty }\geq u\left( 1\right) >u\left(
1\right) -v\left( 1\right) \geq \left( J\underline{f_{\alpha }}\right)
\left( 1\right) ,
\end{equation*}%
a contradiction to our assumption.

Next assume that $Nu>u$ for some $u\in K$ with $\left\Vert u\right\Vert
_{\infty }=\beta $. Then $Nu=u+v$ for some $v\in K\setminus \left\{
0\right\} $ and, since $G(t,s)\leq G(1,s)$ for all $t,\,s\in \lbrack 0,1]$,
we have
\begin{equation}
u\left( t\right) +v\left( t\right) =Jf\left( t,u\left( t\right) \right) \leq
\left( J\overline{f_{\beta }}\right) \left( t\right) \leq \left( J\overline{%
f_{\beta }}\right) \left( 1\right) .  \label{r3}
\end{equation}%
Let $t_{0}$ be such that $u\left( t_{0}\right) =\left\Vert u\right\Vert
_{\infty }=\beta >0.$ Since $u\left( 0\right) =0,$ one has $t_{0}>0$ and so $%
v\left( t_{0}\right) \geq M\left( t_{0}\right) \left\Vert v\right\Vert
_{\infty }>0.$ Then, for $t=t_{0},$ (\ref{r3}) gives%
\begin{equation*}
\beta <\left( J\overline{f_{\beta }}\right) \left( 1\right) ,
\end{equation*}%
which contradicts our assumption. Thus Theorem \ref{tk} applies.

We note that if $\alpha <\beta ,$ then (\ref{r0}) represents the compression
condition, while if $\alpha >\beta ,$ then (\ref{r0}) expresses the
expansion condition.
\end{proof}

Next we are interested into two positive solutions for (\ref{p}). We shall
succeed this by the variational approach based on Theorems \ref{t1} and \ref%
{t2} applied in the Hilbert space $X=\{u\in H^{2}(0,1):$ $\ u(0)=u^{\prime
}(0)=0\}$ and to the norms $\left\vert \cdot \right\vert $ (given by (\ref%
{en})) and $\left\Vert \cdot \right\Vert =\left\Vert \cdot \right\Vert
_{L^{2}\left( 0,1\right) }.$

Let us consider the cone
\begin{equation*}
K=\left\{ u\in X:\ u\text{ convex and }u(t)\geq M\left( t\right)
\,\left\Vert u\right\Vert \,\ \text{for all \ }t\in \lbrack 0,1]\right\}
\end{equation*}%
and the numbers $R_{0},R_{1}$ such that $0<R_{0}<\left\Vert \phi \right\Vert
R_{1},$ where $\phi =\phi _{1}/\left\vert \phi _{1}\right\vert $ and $\phi
_{1}$ is the eigenfunction given by (\ref{ef}). Also, let
\begin{equation*}
K_{R_{0}R_{1}}=\left\{ u\in K:\ \left\Vert u\right\Vert \geq R_{0},\
\left\vert u\right\vert \leq R_{1}\right\} .
\end{equation*}

Denote%
\begin{eqnarray*}
\underline{g}\left( t\right) &=&\min \left\{ f\left( t,u\right) :\ M\left(
t\right) R_{0}\leq u\leq c_{\infty }R_{1}\right\}, \\
\overline{g}\left( t\right) &=&\max \left\{ f\left( t,u\right) :\ M\left(
t\right) R_{0}\leq u\leq c_{\infty }R_{1}\right\} ,
\end{eqnarray*}%
where $c_{\infty }>0$ is such that $\left\Vert v\right\Vert _{\infty }\leq
c_{\infty }\left\vert v\right\vert $ for all $v\in K.$ For example we may
take $c_{\infty }=2/3,$ since for any $v\in K,$ Hölder's inequality gives%
\begin{equation}
v\left( t\right) =\int_{0}^{t}\int_{0}^{s}v^{\prime \prime }\left( \tau
\right) d\tau ds\leq \left\vert v\right\vert \int_{0}^{t}\sqrt{s}\leq \frac{2%
}{3}\left\vert v\right\vert .  \label{r13}
\end{equation}

Our assumptions are as follows: \bigskip

\begin{description}
\item[(H1)] There exist $R_{0},R_{1}$ with $0<R_{0}<\left\Vert \phi
\right\Vert R_{1}$ such that

\item[(a)] $R_{0}\leq \left\Vert J\underline{g}\right\Vert ,$

\item[(b)] $R_{1}\geq c_{\infty }\left\Vert \overline{g}\right\Vert
_{L^{1}\left( 0,1\right) }.$

\item[(H2)] The functional $E$ has the mountain pass geometry in $%
K_{R_{0}R_{1}}$ and there exists $\rho >0$ such that%
\begin{equation}
E\left( u\right) \geq c+\rho  \label{r10}
\end{equation}%
for all $u\in K_{R_{0}R_{1}}$ which simultaneously satisfy $\left\Vert
u\right\Vert =R_{0}$ and $\left\vert u\right\vert =R_{1}.$
\end{description}

\begin{theorem}
Under assumptions \emph{(H1), (H2)}, the problem \emph{(\ref{p})} has at
least two positive solutions $u_{m},u_{c}\in K_{R_{0}R_{1}}$ with $E\left(
u_{m}\right) =m$ and $E\left( u_{c}\right) =c$, with $m$ and $c$ defined on
\emph{\eqref{e-m}} and \emph{\eqref{e-c}} respectively.
\end{theorem}

\begin{proof}
For $u\in K_{R_{0}R_{1}},$ one has%
\begin{equation*}
M\left( t\right) \,R_{0}\leq M\left( t\right) \,\left\Vert u\right\Vert \leq
u(t)\leq \left\Vert u\right\Vert _{\infty }\leq c_{\infty }\,|u|\leq
c_{\infty }\,R_{1}.
\end{equation*}%
It follows that%
\begin{equation*}
F\left( t,u\left( t\right) \right) \leq \omega :=\max \{F\left( t,u\right)
:\ 0\leq t\leq 1,\ M\left( t\right) R_{0}\leq u\leq c_{\infty }R_{1}\},
\end{equation*}%
whence, for all $u\in K_{R_{0}R_{1}}$, it is fulfilled that
\begin{equation*}
E\left( u\right) =\frac{1}{2}\left\vert u\right\vert
^{2}-\int_{0}^{1}F\left( t,u\left( t\right) \right) dt\geq -\omega ,
\end{equation*}%
and so, $m>-\infty $.

Next, from $c>m$ we see that (\ref{r10}) guarantees both (\ref{r11}) and (%
\ref{r12}). It remains to check the compression boundary condition given by (%
\ref{h1}), (\ref{h2}). Assume first that (\ref{h1}) does not hold. Then $%
JE^{\prime }\left( u\right) -\lambda Ju=0\ \ $for some $u\in
K_{R_{0}R_{1}},\ \left\Vert u\right\Vert =R_{0}$ and$\ \lambda >0.$

Then, for $t>0,$
\begin{equation*}
u\left( t\right) =J\left( f\left( t,u\left( t\right) \right) +\lambda
u\left( t\right) \right) >Jf\left( t,u\left( t\right) \right) \geq \left( J%
\underline{g}\right) \left( t\right) \geq 0.
\end{equation*}%
Taking the $L^{2}$-norm, we deduce%
\begin{equation*}
R_{0}=\left\Vert u\right\Vert >\left\Vert J\underline{g}\right\Vert ,
\end{equation*}%
which contradicts (H1)(a).

Next assume that $JE^{\prime }\left( u\right) +\lambda u=0\ \ $for some $%
u\in K_{R_{0}R_{1}},\ \left\vert u\right\vert =R_{1}\ $and $\lambda >0.$ Then%
\begin{equation*}
\left( 1+\lambda \right) u^{\left( 4\right) }=f\left( t,u\left( t\right)
\right) ,
\end{equation*}%
whence, arguing as in the proof of Theorem \ref{t3}, we deduce that
\begin{equation*}
\left( 1+\lambda \right) R_{1}^{2}=\int_{0}^{1}u\left( t\right) f\left(
t,u\left( t\right) \right) dt.
\end{equation*}%
Consequently%
\begin{equation*}
R_{1}<c_{\infty }\left\Vert \overline{g}\right\Vert _{L^{1}\left( 0,1\right)
},
\end{equation*}%
which contradicts (H1)(b).
\end{proof}

\begin{remark}
In the autonomous case, if $f=f\left( u\right) ,$ and $f$ is nondecreasing
on $\mathbb{R}_{+},$ a sufficient condition for (H1)(a) to hold is%
\begin{equation*}
R_{0}\leq f\left( M\left( a\right) R_{0}\right) \left\Vert J\chi _{\lbrack
a,1]}\right\Vert ,
\end{equation*}%
where $a$ is some number from $\left( 0,1\right) $ and $\chi _{\left[ a,1%
\right] }$ is the characteristic function of the interval $\left[ a,1\right]$%
. Also in this case, (H1)(b) reduces to
\begin{equation*}
R_{1}\geq c_{\infty }f\left( c_{\infty }R_{1}\right) .
\end{equation*}
\end{remark}

\subsection{An example}

We are going to see an example inspired by that in \cite{Pre}, to which we
can apply Theorem \ref{t4}.

\begin{example}
Let $0\leq p\leq \frac{1}{2}$ and
\begin{equation}
f(t,u)=f(u)=\left\{
\begin{array}{lc}
p\,u^{p}\,, & 0\leq u\leq 1\,, \\
p\,u^{2}\,, & 1\leq u\leq b\,, \\
p((u-b)^{p}+b^{2})\,, & u\geq b,\,%
\end{array}%
\right.  \label{pp}
\end{equation}%
where $b>2$ is chosen below.

This function is positive and nondecreasing in $\mathbb{R}_{+}.$ For $1\leq
u\leq 2,$ we have
\begin{equation*}
F(u)=\dfrac{p}{p+1}+\dfrac{p}{3}(u^{3}-1)\leq \dfrac{p}{p+1}+\dfrac{7\,p}{3}=%
\dfrac{p\,(10+7\,p)}{3\,(p+1)}\,,
\end{equation*}%
and since $F$ is nondecreasing,%
\begin{equation*}
F\left( u\right) \leq \dfrac{p\,(10+7\,p)}{3\,(p+1)}\,\ \ \ \text{for }0\leq
u\leq 2.
\end{equation*}

Choose $r=2.$ Then, for $u\in K$ and $|u|=2$ according to (\ref{r13}), $%
|u|_{\infty }\leq \frac{2}{3}|u|<2$ and so, recalling $0\leq p\leq \frac{1}{2%
},$
\begin{equation*}
E(u)=\dfrac{|u|^{2}}{2}-\,\int_{0}^{1}F(u(t))\,dt\geq 2-\dfrac{p\,(10+7\,p)}{%
3\,(p+1)}\geq \dfrac{1}{2}\,.
\end{equation*}

We take for $u_{0}$ the normalized function
\begin{equation*}
u_{0}=\dfrac{\phi _{1}}{|\phi _{1}|}\,,\quad |u_{0}|=1<r=2\,,
\end{equation*}%
for which
\begin{equation*}
E(u_{0})=\dfrac{|u_{0}|^{2}}{2}-\int_{0}^{1}F(u_{0}\left( t\right))\,dt=%
\dfrac{1}{2}-\int_{0}^{1}F(u_{0}\left(t\right) )\,dt<\dfrac{1}{2}\,.
\end{equation*}

Next we choose $u_{1}=b\,u_{0}/\left\Vert u_{0}\right\Vert _{\infty }.$ Then
$\left\vert u_{1}\right\vert =b/\left\Vert u_{0}\right\Vert _{\infty }>2$ if
we choose $b>2\left\Vert u_{0}\right\Vert _{\infty }.$ Also%
\begin{equation*}
E(u_{1})\leq \dfrac{b^{2}}{2\left\Vert u_{0}\right\Vert _{\infty }^{2}}%
-\int_{\left( u_{1}>1\right) }F\left( u_{1}\left( t\right) \right) dt.\text{
}
\end{equation*}%
Since $\left\Vert u_{1}\right\Vert _{\infty }=b>2$ and $u_{1}(0)=0$ the
level set $\left( u_{1}>1\right) $ is a proper subset of $\left[ 0,1\right]
. $ Also $u_{1}\left( t\right) \leq b$ for all $t.$ Hence on the level set $%
\left( u_{1}>1\right) $ we have
\begin{equation*}
F\left( u_{1}\right) =\frac{p}{p+1}+\frac{p}{3}\left( u_{1}^{3}-1\right) >%
\frac{p}{3}u_{1}^{3}.
\end{equation*}%
Then
\begin{equation*}
E(u_{1})<\dfrac{b^{2}}{2\left\Vert u_{0}\right\Vert _{\infty }^{2}}-\frac{%
pb^{3}}{3\left\Vert u_{0}\right\Vert _{\infty }^{3}}\int_{\left(
u_{1}>1\right) }u_{0}\left( t\right) ^{3}dt.
\end{equation*}%
Taking into account that the level set $\left( u_{1}>1\right) $ enlarges as $%
b$ increases, we can see that the right side of the last inequality tends to
$-\infty $ as $b\rightarrow +\infty .$ Thus we may choose $b$ large enough
to have
\begin{equation*}
E(u_{1})<\dfrac{1}{2}.
\end{equation*}%
Hence the assumption (h3) of Theorem \ref{t4} is satisfied. Also, by
\begin{equation*}
\lim_{s\rightarrow 0}\dfrac{f(s)}{s}=+\infty \,,\quad \text{and}\quad
\lim_{s\rightarrow +\infty }\dfrac{f(s)}{s}=0\,,
\end{equation*}%
we may find $R_{0}$ (small enough) and $R_{1}$ (large enough), such that $%
u_{0}$ and $u_{1}$ belong to $K_{R_{0}R_{1}}$ \ and the conditions (\ref{f2}%
) hold. Therefore, according to Theorem \ref{t4}, the problem (\ref{p}) with
$f$ given by (\ref{pp}) and $b$ sufficiently large has two positive
solutions.
\end{example}

\section*{Acknowledgements}

The first and third authors were partially supported by Ministerio de Educaci%
ón y Ciencia, Spain, and FEDER, Project MTM2013-43014-P.

The second author was supported by a grant of the Romanian National
Authority for Scientific Research, CNCS -- UEFISCDI, project number
PN-II-ID-PCE-2011-3-0094.

The third author is supported by a FPU scholarship, Ministerio de Educaci\'on, Cultura y Deporte, Spain.

The fourth author is thankful to the Department of Mathematical Analysis, 
University of Santiago de Compostela, Spain, where a part of the paper was prepared,
during his visit.

\medskip

\textit{Authors' addresses:}

\bigskip

Alberto Cabada\newline
Departamento de An\'alise Matem\'atica\newline
Facultade de Matem\'aticas\newline
Universitade de Santiago de Compostela\newline
Santiago de Compostela, Spain\newline
e--mail address: \verb|alberto.cabada@usc.es|

\bigskip

Radu Precup\newline
Department of Mathematics,\newline
Babe\c{s}-Bolyai University, \newline
400084 Cluj-Napoca, Romania, \newline
e--mail address: \verb|r.precup@math.ubbcluj.ro|

\bigskip

Lorena Saavedra\newline
Departamento de An\'alise Matem\'atica\newline
Facultade de Matem\'aticas\newline
Universitade de Santiago de Compostela\newline
Santiago de Compostela, Spain\newline
e--mail address: \verb|lorena.saavedra@usc.es|

\bigskip

Stepan Tersian\newline
Department of Mathematics \newline
University of Ruse\newline
7017 Ruse, Bulgaria\newline
e--mail address: \verb|sterzian@uni-ruse.bg|

\end{document}